\documentclass[10pt]{amsart}
\usepackage{amsmath}
\usepackage{amssymb,amscd}
\usepackage{amscd}
\usepackage{amsmath,latexsym,amssymb,amsthm}
\usepackage{epsfig}
\usepackage{graphicx}
\usepackage{color}
\definecolor{Blue}{rgb}{0.3,0.3,0.9}

\numberwithin{equation}{section}
\newtheorem{secn}{Definition}[section]
\newtheorem{thm}[secn]{Theorem}
\newtheorem{cor}[secn]{Corollary}
\newtheorem{prop}[secn]{Proposition}
\newtheorem{lem}[secn]{Lemma}
\newtheorem{defn}[secn]{Definition}

\newcommand\CC{{\mathcal C}}

\newcommand\FF{{\mathcal F}}
\newcommand\GG{{\mathcal G}}
\newcommand\HH{{\mathcal H}}

\newcommand\LL{{\mathcal L}}
\newcommand\MM{{\mathcal M}}

\newcommand\PP{{\mathcal P}}

\newcommand\TT{{\mathcal T}}

\newcommand\TC{{\TT\CC}}

\newcommand\PMF{{\PP\kern-2pt\MM\FF}}

\newcommand\PML{{\PP\kern-2pt\MM\LL}}

\newcommand\hhat{\widehat}

\newcommand{\fsubd}{\mathrel{{\scriptstyle\searrow}\kern-1ex^d\kern0.5ex}}
\newcommand{\bsubd}{\mathrel{{\scriptstyle\swarrow}\kern-1.6ex^d\kern0.8ex}}
\newcommand{\fsubeq}{\mathrel{\raise-.7ex\hbox{$\overset{\searrow}{=}$}}}
\newcommand{\bsubeq}{\mathrel{\raise-.7ex\hbox{$\overset{\swarrow}{=}$}}}

\newcommand{\bbar}{\overline}

\newcommand{\tsh}[1]{\left\{\kern-.9ex\left\{#1\right\}\kern-.9ex\right\}}

\begin{document}
\title{Height in Splittings of Relatively Hyperbolic Groups}
\author{Abhijit Pal}
\address{Department of Mathematics \& Statistics, IIT-Kanpur,208016}
\footnote{Research Supported by DST-INSPIRE Grant}
\email{abhipal@iitk.ac.in}

\begin{abstract}
Given a finite graph of relatively hyperbolic groups with its fundamental group relatively hyperbolic and 
 edge groups quasi-isometrically embedded and relatively quasiconvex in vertex groups, we prove that  vertex groups are relatively quasiconvex
 if and only if all the vertex groups have finite relative height in the fundamental group.
 \end{abstract}
 \maketitle
  

\section{Introduction}
 A subgroup $H$ of a word hyperbolic group $G$ is said to have \textit{height} 
$n$, if there exists $n$ distinct cosets ${g_1}H,...,{g_n}H$ such that $\bigcap\limits_{i=1}^ng_iH{g_i}^{-1}$
is infinite and $n$ is maximal possible. In \cite{GMRS},  Gitik, Mitra, Rips and Sageev proved that a quasiconvex subgroup of a word hyperbolic group has finite height (Theorem 2.2 of \cite{GMRS}).
Converse of this result was asked by Swarup and it is still an open problem. As a partial converse, in \cite{height}, Mitra proved the following theorem:
\begin{thm}\label{mitra thm}(Theorem 4.6 of \cite{height}) Suppose $G$ is a  word hyperbolic group  which splits over
a subgroup $H$ (i.e. $G=G_1*_HG_2$ or $G=G_1*_H$) with vertex groups and edge groups word hyperbolic and
inclusions of $H$ in $G_1,G_2$ are quasi-isometric embeddings. Then, $H$ has finite height if and only if $H$ is quasiconvex
in $G$.
\end{thm}
For a relatively hyperbolic group, if we continue to have the same definition of height of subgroups as above then a relatively quasiconvex subgroup may not have finite height. For instance, if some group $G$ is hyperbolic relative to $\mathbb Z\oplus\mathbb Z$ then 
the subgroup $\mathbb Z$ ($\hookrightarrow \mathbb Z\oplus\mathbb Z$) is relatively quasiconvex but has infinite height in $G$ according to the definition of height of subgroups as above. An example of  a group $G$ hyperbolic relative to $\mathbb Z\oplus\mathbb Z$ is $\mathbb Z*(\mathbb Z\oplus\mathbb Z)$ or can be obtained from a hyperbolic $3$-manifold fibering over circle with fiber a punctured hyperbolic surface (See Theorem 4.9 of \cite{Mj com}).
The notion of height for a subgroup of a relatively hyperbolic group was first given by Hruska and Wise in \cite{wise}.
To distinguish the notion of height of subgroups in a relatively hyperbolic group from that of  a word hyperbolic group, it is renamed and  now commonly known as relative height of subgroups (see Definition 7.1 of \cite{GM} due to Groves and Manning).
A subgroup $H$ of a relatively hyperbolic group $G$  is said to have \textit{relative height} $n$, if there exists $n$ distinct cosets ${g_1}H,...,{g_n}H$ such that $\bigcap\limits_{i=1}^ng_iH{g_i}^{-1}$
contains a loxodormic element  and $n$ is maximal possible. In \cite{wise},
Hruska and Wise proved that a relatively quasiconvex subgroup $H$ of a relatively hyperbolic group $G$
has  finite relative height. Converse of this result is still not known. In this article, we will generalize Mitra's Theorem \ref{mitra thm}
for a finite graph of relatively hyperbolic groups.

\begin{thm}(Theorem \ref{main})
Let  $G$ be the fundamental group of a finite  graph of groups  $(\GG,\Lambda)$
satisfying the following conditions:\\
(i) for each vertex $v$ of the finite graph $\Lambda$, the vertex group $G_v$ is finitely generated and hyperbolic relative to a 
finite collection $\HH_v$ of finitely generated subgroups of $G_v$;\\
(ii) the edge groups are quasi-isometrically embedded and relatively quasiconvex in respective vertex groups,  \\
(iii) $G$ is hyperbolic relative to  $\HH$, where $\HH$ is the collection of images of all $\HH_v$ in $G$,\\
(iv) for each vertex $v$ of $\Lambda$, $G_v$ has finite relative height in $G$.\\
Then for every vertex $w$ of $\Lambda$ the group $G_w$ is a relatively quasiconvex subgroup of $G$.
\end{thm}

 Let $G$ be a word hyperbolic group such that $G=A*_CB$
where $A,B$ are hyperbolic groups and $C$ is quasiconvex in both $A,B$.
Suppose $G$ acts acylindrically on the Bass-Serre tree i.e. there exists a constant $k\geq 1$ such that for every geodesic segment in the  Bass-Serre tree of length $\geq k$, the pointwise
stabilizer of that segment in $G$ is finite. I.Kapovich in \cite{kapo} showed that  for a
finite graph of hyperbolic groups with edge groups quasi-isometrically embedded in vertex groups if the fundamental group of graph of groups
acts acylindrically on the Bass-Serre tree then all vertex groups are quasiconvex in the fundamental group. The techniques used by I.Kapovich in \cite{kapo} are completely different from that
of Mitra's  \cite{height}. An analogous version of I.Kapovich's result for relatively hyperbolic groups was given by Dahmani in \cite{dah}.
A subgroup of a relatively hyperbolic group is called \textit{fully quasiconvex} if it is relatively quasiconvex and 
only finitely many translates of its limit set in the relative hyperbolic boundary can intersect altogether (Definition 1.6 of \cite{dah}, see also Definition 6.2 of \cite{hrus}).
In \cite{dah}, Dahmani proved that for a
finite graph of relatively hyperbolic groups with edge groups fully quasiconvex in vertex groups if the fundamental group of graph of groups acts acylindrically on the Bass-Serre tree then the
vertex groups are fully quasiconvex in the fundamental group of graph of groups (See Corollary 4.1 of \cite{dah}). Dahmani's method of proof involved using the action of vertex groups on their 
relatively hyperbolic boundaries. Our approach to the proof of Theorem \ref{main} relies heavily on the 
machineries developed in \cite{pal} and does not involve group action on boundaries. 

A source of examples of finite graphs of relatively hyperbolic groups with its fundamental group also relatively hyperbolic comes from the Combination Theorem due to Mj. and Reeves (See Theorem 4.6 of \cite{Mj com}). In the word-hyperbolic setting, the Combination Theorem in same flavor is due to Bestvina and Feighn and it states that for a finite graph of word hyperbolic groups with edge groups quasi-isometrically embedded in vertex groups if ‘hallways flare condition’ is satisfied then the fundamental group of the finite graph of word hyperbolic groups is word hyperbolic (See Section 4 of \cite{best}).
The Mj.-Reeves's Combination Theorem for a finite graph of relatively hyperbolic groups assumes conditions similar to item (ii) of Theorem \ref{main} and certain ‘flaring’ conditions (see conditions (4) \& (5) of Theorem 4.6 of \cite{Mj com})   so that  the fundamental group of the graph of relatively hyperbolic groups is also relatively hyperbolic.

Let $G=A*_CB$ where $G,A,B,C$ are word hyperbolic groups and  $A,B,C$ are quasiconvex in $G$.
Using Mitra's Theorem \ref{mitra thm}, Swarup (Theorem 2 of \cite{swarup}) proved that  a finitely generated subgroup $Q$ of $G$ is quasiconvex if and only if 
the intersection of $Q$ with every conjugates of $A,B,$  in $G$ are quasiconvex in $G$.
 As an application of Theorem \ref{main}, 
we will give a natural generalization of this result to relatively hyperbolic case (See Theorem \ref{amalg thm}).

\section{Relatively Hyperbolic Groups}\label{rhg}

 \subsection{Relative Hyperbolicity}
Relatively hyperbolic groups were first introduced by Gromov 
\cite{gro} to study hyperbolic manifolds with cusps. The notion of relative hyperbolicity  was then studied by several people, we refer to the article \cite{hrus} by 
Hruska for several equivalent notions of relatively hyperbolic groups. Here we give two equivalent definitions of 
relatively hyperbolic groups due to Farb \cite{farb} and Bowditch \cite{bow}. 
\begin{defn}(Hyperbolic Metric Space)\label{geod hyp}
  Let $\delta\geq 0$. We say that a geodesic triangle $\Delta$ is $\delta$-\textit{slim } in a geodesic metric space
 if any side of the triangle $\Delta$ is contained in the $\delta$- neighborhood of the union of the other two sides.
A geodesic metric space is said to be $\delta$-\textit{hyperbolic} if all the triangles are $\delta$-slim. 
A geodesic metric space is said to be hyperbolic if it is $\delta$-hyperbolic for some $\delta\geq 0$.
\end{defn}

First we give the definition of relative hyperbolicity due to Farb \cite{farb}. Let $G$ be a finitely generated group and $\HH$ be a 
finite collection of finitely generated subgroups of $G$. Let $\Gamma_G$ be the Cayley graph of $G$ with respect to some finite generating set.

\begin{defn}{(Coned-off Cayley Graph, Definition 3.1 of \cite{farb})}
 The coned-off Cayley graph of $G$ with respect to $\HH$, denoted by $\hhat\Gamma_G$, 
is obtained from $\Gamma_G$ by adding an extra vertex $v(gH)$ for each left coset  $gH$ corresponding to each $H\in \HH$ and an
extra edge $[gh,v(gH)]$ of length 1/2 joining each $gh\in gH$ to $v(gH)$. 
\end{defn}

\noindent For a path $\gamma :[0,1]\to \Gamma_G$, there is an induced path $\hhat{\gamma}$
in $\hhat\Gamma_G$ obtained as follows:\\
for each maximal component $\gamma|_{[s,t]}$ of $\gamma$ lying in a coset $gH$,
replace the component $\gamma|_{[s,t]}$  by the path $[\gamma(s),v(gH)]\cup[v(gH),\gamma(t)]$
of length one passing through cone point $v(gH)$.\\
If $\hhat{\gamma}$ is a geodesic in $\hhat\Gamma_G$, 
$\gamma$ is called  {\em relative geodesic}. If $\hhat{\gamma}$ is a $P$-quasi-geodesic, $\gamma$ is called $P$-{\em relative quasi-geodesic}.\\
For a path $\gamma$ in $\hhat\Gamma_G$, we say that $\gamma$ {\em penetrates} a coset $gH$ if $\gamma$ passes through the cone point $v(gH)$
and we say that $\gamma$ is  {\em without backtracking } if whenever $\gamma$  penetrates some coset $gH$, the path $\gamma$ does not return to the coset $gH$ after leaving this coset.\\

\begin{defn}{(Bounded Coset Penetration Property, See section 3.3 of \cite{farb})}
The pair $(G,\HH)$ is said to have bounded coset penetration property if for each $k>1$ there exists $c(k)>0$ such that for any two relative $k$-quasi geodesics
$\gamma_{1}, \gamma_{2}$ in $\Gamma_{G}$ with same end points and for any $H\in \HH$, the following holds,

(1)  if $\gamma_{1}$ penetrates $gH$  but $\gamma_{2}$ does not then $\gamma_{1}$ travels at most $c(k)$ distance in $gH$.
\end{defn}

(2) if both $\gamma_{1}, \gamma_{2}$ penetrate $gH$ then the entry points as well as the exit points of the paths are $c(k)$ close to each other in $\Gamma_{G}$.

\begin{defn}\label{farb}{(B.Farb, See section 5 of \cite{farb})}
Let $\delta\geq 0. $Let $G$ be a finitely generated group and $\HH$ be a finite collection of finitely generated subgroups of it. 
The group $G$ is said to be $\delta$-hyperbolic relative to $\HH$ if $\hhat\Gamma_G$ is $\delta$-hyperbolic and the pair $(G,\HH)$ satisfies the bounded coset penetration property.
The group $G$ is said to be hyperbolic relative to $\HH$ if it is $\delta$-hyperbolic relative to $\HH$ for some $\delta\geq 0$. Conjugates of elements
of $\HH$ are called parabolic subgroups.
\end{defn}
 
Let $G$ be a group hyperbolic relative to $\HH$. An element $g\in G$ is called \textit{parabolic} if $g$ has 
infinite order and belongs to a conjugate of some $H\in\HH$.
An element $g\in G$ is called \textit{elliptic} if $g$ has finite order.
 An element $g\in G$ is called \textit{loxodormic} if $g$ has infinite order and is not parabolic.\\
The next definition by Bowditch gives a dynamical characterization of relative hyperbolicity.
For that reason we need the notion of convergence group.

\begin{defn}{(Convergence Group)}
Let $G$ be a group  acting on a compact metrizable space $M$. The action is called  convergence group action if 
for any sequence $\{g_{n}\}$ in $G$, there exists a subsequence $\{g_{\phi(n)}\}$ and $\xi^{+},\xi^{-}\in M$ such that $g_{\phi(n)}(K)$ 
converges uniformly to $\xi^{+}$, for all compact sets  $K\subset M\backslash \{\xi^{-}\}$.
\end{defn}

\begin{defn}\begin{enumerate}
 \item {(Bounded Parabolic Limit Points)}
An element $g\in G$ is called parabolic if it fixes exactly one point of $M$ and the corresponding fixed point $\xi$(say) is said to be
parabolic limit point. Furthermore, a parabolic limit point is said to be bounded parabolic if $Stab(\xi)$ acts properly discontinuously 
and cocompactly on $M\backslash \{\xi\}$.
\item {(Conical Limit Point)} Let $G$ be a group with a convergence action on $M$. A point $\xi\in M$ is said to be conical limit point if there exists a sequence 
$\{g_{n}\}$ and $\xi^{+}\neq\xi^{-}\in M$ such that 
$g_{n}\xi \rightarrow \xi^{+}$, $g_{n}\xi^{'} \rightarrow \xi^{-}$ for all $ \xi^{'}\in M\backslash \{\xi\}$.
\item {(Geometrically Finite Action)} Let $G$ be a group with a convergence action on a compact metrizable space $M$. The action is said to be 
geometrically finite if the limit points are either conical or bounded parabolic.
             
     \end{enumerate}

\end{defn}

\begin{defn}{(Bowditch, Definition 1 of \cite{bow})} \label{bowditch}
Let $G$ be a finitely generated group and $\HH$ be a finite collection of finitely generated subgroups of it. 
The group $G$ is said to hyperbolic relative to $\HH$ if it acts properly discontinuously by isometries on a proper hyperbolic metric space $\widetilde{\Gamma}$ such that
\begin{itemize}
  \item the induced action of the group $G$ on $\partial \widetilde{\Gamma}$ is  convergence and  geometrically finite.
  \item the conjugates of the elements of $\HH$ are precisely the maximal parabolic subgroups.
\end{itemize}
we call $\partial \widetilde{\Gamma}$ the Bowditch boundary of $G$.
\end{defn}
 
The Definition \ref{farb} is equivalent to the Definition \ref{bowditch} (See Theorem 5.1 of \cite{hrus}). The two notions of parabolic elements in a relatively hyperbolic group given above are also  equivalent.

\begin{defn}\label{relqc}(Relatively Quasiconvex Subgroup):
Let $q\geq 0$. Suppose $G$ is a finitely generated group hyperbolic relative to a finite collection 
$\HH$ of finitely generated subgroups of $G$.
A subgroup $K$ of $G$ is said to be $q$-relatively quasiconvex if for any two points $x,y\in K$
the geodesics in the coned-off graph $\hhat\Gamma_G$ joining $x,y$ lie in the $q$-neighborhood of $K$ (where the neighborhood is taken in the coned-off graph $\hhat\Gamma_G$).

\end{defn}

\begin{defn}(Relative Height, Definition 7.1 of \cite{GM}, Section 1.3 of \cite{wise})\label{height defn}
Suppose $G$ is hyperbolic relative to a finite collection $\HH$ of finitely generated subgroups and $H$ is a subgroup
of $G$. The \textit{relative height} of $H$ in $G$ is the maximum number
$n\geq 0$ so that there are $n$ distinct cosets ${g_1}H,...,{g_n}H$ such that $\bigcap\limits_{i=1}^ng_iH{g_i}^{-1}$ contains a loxodormic element 
 (or equivalently, $\bigcap\limits_{i=1}^n g_iH{g_i}^{-1}$  is an infinite non-parabolic subgroup).\end{defn}

\begin{thm}(Theorem 1.4 of \cite{wise})\label{finite relative height}
Let $G$ be a relatively hyperbolic group, then relative height of a relatively quasiconvex group in $G$ is finite.
\end{thm}

\subsection{Graph of Groups, Trees of Spaces}
A graph $\Lambda$ consists of a vertex set $V=Vertex(\Lambda)$, edge set $E=Edge(\Lambda)$ and two maps given by
$$E\to V\times V,~ e\mapsto (o(e),t(e)),$$ and 
$$E\to E, e\mapsto\bar e$$
satisfying the following conditions : $\bar{\bar e}=e,\bar e\neq e$ and $o(e)=t(\bar{e})$. An element $v\in V$ is called a \textit{vertex} of $\Lambda$, an element $e\in E$ is called an \textit{(oriented) edge} of $\Lambda$, and $\bar e$ is called the \textit{inverse} edge. The vertex $o(e)$ is called the \textit{origin} of $e$ and the vertex $t(e)$ is called the \textit{terminus} of $e$. For more details, see 
 Section 2.1 of \cite{serre}.
\begin{defn}(Serre, Definition 8 of \cite{serre})
 A graph of groups $(\GG,\Lambda)$ consists of the following:
 \begin{enumerate}
  \item a connected graph $\Lambda$,
  \item a group $G_v$ for each vertex $v\in V$,
  \item a group $G_e$ for each edge $e$ of $\Lambda$ such that $G_{\bar e}=G_{e}$ where $\bar e$ is reverse of $e$,
  \item monomorphisms $G_e\to G_{t(e)}$ denoted by $a\mapsto a^e$ for each edge $e$ of $\Lambda$.
 \end{enumerate}  
\end{defn}
Let $(\GG,\Lambda)$ be a graph of groups.
For each vertex $v$ of $\Lambda$, let $<U_v|R_v>$ be a presentation of the vertex group $G_v$ .
Let $U=\sqcup_{v\in V}U_v$ and $R=\sqcup_{v\in V}R_v$.
Consider the free group $\mathbb F(U \cup E)$ with generating set $U\cup E$.
For an edge $e\in E$ and an element $ a\in G_e$, consider the elements $P_e=\bar e e$, $Q_{a,e}=ea^ee^{-1}a^{-\bar e}$ in $\mathbb F(U \cup E)$. Let $F(\GG, \Lambda)$ be the group with presentation $<U \cup E|P\cup Q\cup R>$ where $P=\{P_e:e\in E\}$ and 
$Q=\{Q_{a,e}: a\in G_e,e\in E\}$. 
Thus,  $F(\GG,\Lambda)$ is a group generated by the groups $G_v$ ($v\in V$) and the 
elements $e\in E$ subject to the relations:
$$\bar e=e^{-1}~\mbox{and}~ea^ee^{-1}=a^{\bar e}$$
where $e\in E, a\in G_e$.

\begin{defn}(Fundamental Group, Section 5.1 of Serre \cite{serre}) 
Let $(\GG,\Lambda)$ be a graph of groups. Let $S$ be a maximal tree of $\Lambda$. The fundamental group $\pi_1(\GG,\Lambda,S)$ of graph of groups $(\GG,\Lambda)$ is defined
to be the quotient of $F(\GG,\Lambda)$ by the normal subgroup generated by the elements $e\in Edge(S)$. 
 \end{defn}
 If $x_e$ denotes the image of an edge $e$
in $\pi_1(\GG,\Lambda,S)$, then $\pi_1(\GG,\Lambda,S)$ is the group generated by $G_v$ ($v\in V$) and  
elements $x_e$ ($e\in E$) subject to the relations:
$$x_ea^ex^{-1}_e=a^{\bar e}, x_{\bar e}=x^{-1}_e~\mbox{if}~e\in E,a\in G_e,$$
$$x_e=1~\mbox{if}~e\in Edge(S).$$
 
Corresponding to the fundamental group $G=\pi_1(\GG,\Lambda,S)$, there is a tree $T$, called Bass-Serre tree, on which the group $G$ acts without edge-inversions such that the quotient space $T/G$ of $T$ is the finite graph $\Lambda$.
The vertex set of  $T$ is $\{gG_v:g\in G,v\in V\}$
and edge set is $\{gG_e:e\in E\}$, where $V=Vertex(\Lambda)$, $E=Edge(\Lambda)$. 
The incidence relation depends upon the orientation of the graph $\Lambda$. Fix an orientation of the graph $\Lambda$. For  an edge $e$ having orientation same
as that of induced from the graph $\Lambda$, define $o(gG_e):=gG_{o(e)}$ and $t(gG_e):=gx_eG_{t(e)}$; otherwise
define $o(gG_e):=gx^{-1}_eG_{o(e)}$ and $t(gG_e):=gG_{t(e)}$. Define $\bbar{gG_e}:=gG_{\bar e}$ for any edge $e\in E$.
We refer to the Section 5.3 of \cite{serre} for details.
Vertex stabilizer of a vertex in $T$ is conjugate to a vertex group $G_w$ for some $w\in V$. Similarly, the same holds for every edge stabilizer.

Suppose $(\GG,\Lambda)$ is a finite graph of groups with all vertex and edge groups finitely generated.
Let $\Gamma_v,\Gamma_e$ denote the Cayley graphs of $G_v,G_e$ respectively with edge lengths one, where $v\in V$ and $e\in E$.
Consider the following disjoint union :
$$\mathbb X=\underset{v\in V,g\in G}\sqcup (g\Gamma_v)\sqcup(\underset{e\in E,g\in G}\sqcup(g\Gamma_e\times [0,1])).$$
Let $X$ be the quotient of $\mathbb X$ obtained from the equivalence relations as follows:\\
(i) for every edge $e\in E$ having orientation same
as that of induced from the graph $\Lambda$,
\[ (ga,0)\sim (ga^e,0), (ga,1)\sim (gx_ea^e,1,)\]\\
(ii) for every edge $e\in E$ having orientation opposite
to that induced from the graph $\Lambda$
\[ (ga,0)\sim (gx^{-1}_ea^e,0), (ga,1)\sim (ga^e,1),\]
where $g\in G, a\in\Gamma_e$.\\
There is a natural map $P:X\to T$ such that $P(gg_v)=gG_v $ and $ P((ga,t))=x_t$ where $x_t$ is the point on the edge $gG_e$ at a distance $t$ from $o(gG_e)$, $g_v\in \Gamma_v,g\in G$ and $a\in\Gamma_e$.
For a vertex $u$ in $T$, $X_u=P^{-1}(u)$ will be called the vertex space for $u$.  
Let $e$ be an edge of $T$ and $X_e$ be the pre-image under $P$ of the mid-point of  $e$, $X_e$ will be called the edge space for $e$.
For each vertex $u$ of $T$, $P^{-1}(u)$ is isometric to a Cayley graph $\Gamma_v$ for some vertex $v$
of $\Lambda$.  Metric on $X$ is the path metric induced from
the individual metrics on vertex and edge spaces.  Thus, we have a metric space $X$ 
admitting a map $P:X\to T$ onto the simplicial tree $T$ such that the fiber to each vertex of $T$ is a geodesic space. In short, we will call $P:X\to T$ to be as \textit{tree of spaces}.

For every vertex $v$ and edge $e$ of $\Lambda$, the spaces $P^{-1}(gG_v)$, $P^{-1}(gG_e)$ are stabilized by subgroups $gG_vg^{-1}$, $ gG_eg^{-1}$ respectively.
These stabilizer subgroups, being finitely generated, act properly and cocompactly on their respective vertex and edge spaces.
The action of $G$ on $X$ is proper and  also co-compact as $\Lambda$ is a finite graph. By applying \u{S}varc-Milnor
lemma, the Cayley graph $\Gamma_G$ of $G$ is quasi-isometric to $X$. \\

For a finite  graph of groups  $(\GG,\Lambda)$ we assume that the following conditions hold:\\
(i) for each vertex $v$ of the finite graph $\Lambda$, the vertex group $G_v$ is finitely generated and hyperbolic relative to a 
finite collection $\HH_v$ of finitely generated subgroups of $G_v$;\\
(ii) the edge groups are  quasi-isometrically embedded and relatively quasiconvex in respective vertex groups.   \\
(iii) $G$ is hyperbolic relative to  $\HH$, where $\HH$ is the collection of images of all $\HH_v$ in $G$.

Condition (i) implies that every vertex space in the tree of spaces $P:X\to T$ is relatively hyperbolic. As $\Lambda$ is a finite graph,
there exists a $\delta\geq 0$ such that every coned-off vertex space is $\delta$-hyperbolic.\par
It is a standard fact in geometric group theory that a quasi-isometrically embedded subgroup of a finitely generated group
is also finitely generated. Conditions (i) and (ii) hence imply  that all edge groups are finitely generated. 
 As $\Lambda$ is a finite graph and
all local groups are finitely generated, $G$ is a finitely generated group. Condition (ii) further implies that there exists $k\geq 1$ and $\epsilon\geq 0$ such that for every edge $e$ of $T$ with end points $v_1,v_2$ there is a
$(k,\epsilon)$-quasi-isometric embedding ${f_{e,v_i}}: X_e\to X_{v_i}, i=1,2$. 
  For each edge group $G_e$, we have assumed that $G_e$ is relatively quasiconvex in $G_{o(e)},G_{t(e)}$ respectively.
Taking intersection of parabolic subgroups (in vertex groups) with the images of edge groups $G_e$,
we will get a finite collection $\HH_e$ of subgroups of $G_e$ such that $G_e$ is hyperbolic relative to $\HH_e$. The monomorphisms $G_e\hookrightarrow G_{o(e)}$, $G_e\hookrightarrow G_{t(e)}$ \textit{preserve cusps} i.e.
the images of a parabolic subgroup of $G_e$ under the maps $G_e\hookrightarrow G_{o(e)}$, $G_e\hookrightarrow G_{t(e)}$  will be again a parabolic subgroup in $G_{o(e)},G_{t(e)}$.  The monomorphisms $G_e\hookrightarrow G_{o(e)}$, $G_e\hookrightarrow G_{t(e)}$  are \textit{ strictly type preserving} i.e. the inverse images of the parabolic subgroups of $G_{o(e)},G_{t(e)}$ under the maps $G_e\hookrightarrow G_{o(e)}$, $G_e\hookrightarrow G_{t(e)}$  respectively are either empty or  parabolic subgroups of $G_e$.
 Thus, we have a tree of relatively hyperbolic spaces as per Definition 3.1 of \cite{Mj com}. Coning-off  all vertex spaces and edge spaces we have a tree of coned-off
spaces, denoted by $\TC(X)$, where the tree is same as $T$ (See Page 1787 of Section 3 of \cite{Mj com} for tree of coned-off spaces). \par
Condition (iii) implies that $\TC(X)$ is a hyperbolic metric space.\\
The surjective map $P:X\to T$ naturally induces a surjective map 
$\hhat P:TC (X) \rightarrow T$. For the sake of notation, we will denote the tree of coned-off spaces $\TC(X)$ with $\hhat P:TC (X) \rightarrow T$ as $P: \TC (X) \rightarrow T$. 
 For every vertex $v$ and edge $e$ of $T$, we shall denote coned-off vertex and edge spaces as $\hhat{ G_v}$ and 
$\hhat {X_e}$ respectively. 
\subsection{Hyperbolic Ladder and Retraction map}
\label{subsec:cons}
Let $P:X\to T$ be a tree of $\delta$-relatively hyperbolic spaces.
Given a geodesic segment $\hhat\lambda \subset \hhat {X_{v_0}}$ with end points lying outside cosets of parabolic subgroups. 
In \cite{pal} a quasiconvex set $B_{\hhat \lambda} \subset \TT\CC(X)$ was constructed and it was called to be the hyperbolic ladder (See Section 2.1 of \cite{pal}).
 We will not give here the construction of the ladder $B_{\hat \lambda}$, instead we will extract the important
properties of the ladder which will be required for our purpose.

\smallskip

\noindent {\textbf{Hyperbolic Ladder $B_{\hat\lambda}$}}\\

 $B_{\hhat\lambda}$ is a collection of paths in vertex spaces of $\TT\CC(X)$ satisfying the following properties:
\begin{enumerate}
 \item (Non-empty)\label{non-empty} For each $v\in Vertex (T)$, $B_{\hat\lambda}\cap\hhat G_v$ is either empty or a geodesic segment in $\hhat G_v$ denoted as $\hhat\lambda_v$.
 If $v=v_0$, then $\hhat\lambda_v=\hhat\lambda$.
 \item (Flow) \label{Flow} There exists a constant $C>0$ (depending only on $\delta$) such that the following happens:\\
 Let $v_1$ and $v_2$ be two adjacent vertices of $T$ joined by an edge $e$ and let $\lambda_{v_1}^b=\hhat\lambda_{v_1}\cap X_{v_1}$ for a 
 geodesic $\hhat\lambda_{v_1}$ of  $\hhat X_{v_1}$. If $B_{\hat\lambda}\cap\hhat X_{v_1}=\hhat\lambda_{v_1}$ and $Nbhd_{G_v}(\lambda_{v_1}^b;C)\cap f_{e,v_1}(X_e)\neq\phi$,
 then $B_{\hat\lambda}\cap\hhat X_{v_2}\neq\phi$.
 \item (Retraction Map, See Section 2.2 of \cite{pal})\label{Retraction Map} Let $T_1$ be subtree of $T$ spanned by $P(B_{\hhat \lambda _{v_0}})$. There exists a retraction map 
 $\hhat\pi_{\hhat\lambda}:\TT\CC(X)\to B_{\hhat\lambda}$ such that if $v\in Vertex(T_1)$ then for all $x\in\hhat G_v$,
 $\hhat\pi_{\hhat\lambda}(x)$ is a nearest point projection of $x$ onto $\hhat\lambda_v$.
 \item (Quasiconvexity\label{Quasiconvexity}, Theorem 2.2 of \cite{pal}) There exists a constant $C_0 \geq 0$ such that 
\begin{center}
$d_{\hhat X}(\hhat \Pi_{\hhat \lambda}(x), \hhat \Pi_{\hhat \lambda}(y))\leq C_0 d_{\hhat X}(x, y)+C_0$ for $x, y\in \hhat X$.  
\end{center}
If $\TT\CC(X)$ is $\delta$-hyperbolic, then $B_{\hhat \lambda}$ is quasiconvex, where the quasi-convexity constant depends only on $\delta$ 
and not on the geodesics $\hhat\lambda$.
\item \label{qi}(Quasi-isometric Sections, Lemma 2.5 of \cite{pal}) There exists a constant $\rho\geq 0$ such that for each $v\in Vertex(T_1)$ and for each 
$x\in\lambda^b_v$, where $\lambda^b_v=\hhat\lambda_v\cap X\subset B_{\hhat \lambda}\cap X$, there exists a  map 
$r_x\colon S \to B_{\hhat\lambda}\cap X$ such that $r_x(v)=x$ 
and $d_S(u,u')\leq d_X(r_x(u), r_x(u')\leq \rho d_S(u, u')$, where $S=[v,v_0]$ denotes the geodesic edge  path 
in $T_1$ joining $v$ and $v_0$ and $u,u'\in S$.  

\end{enumerate}

\begin{lem}\label{infinite diameter}
Let $P:X\to T$ be a tree of $\delta$-relatively hyperbolic spaces and $X_{v_0}$ be a vertex space of it.  If $\hhat{ X_{v_0}}$ is not quasiconvex in $\TT\CC(X)$, then there exists a sequence of relative geodesics $\lambda_i$ in $X_{v_0}$
 such that $diam_T(P(B_{\hat\lambda_i}))\to\infty$ as $i\to\infty$.
\end{lem}
\begin{proof}
 $\hhat{ X_{v_0}}$ not quasiconvex in $\TT\CC(X)$ implies that there exists an infinite sequence $\lambda_i$ of distinct relative geodesics in $X_{v_0}$
 such that the geodesics $\alpha_i$ of $\TT\CC(X)$ joining end points of $\lambda_i$ does not lie entirely inside $i$-neighborhood of $\hhat\lambda_i$.
 Therefore, there exists $x_i\in\alpha_i$ such that $d_{\TT\CC(X)}(x_i,\hhat\lambda_i)\geq i$. As $B_{\hat\lambda_i}$
 is quasiconvex, $\alpha_i$ lie in a uniform bounded neighborhood of $B_{\hhat\lambda_i}$.
 Thus, $d_{\TT\CC(X)}(\hhat \Pi_{\hhat\lambda_i}(x_i),\hhat\lambda_i)>i-K$. for some constant $K>0$ independent of $i$. This inequality implies that there exists 
 a vertex $v_i$ such that $\hhat \Pi_{\hhat\lambda_i}(x_i)\in \hhat {X_{v_i}}$ and $d_T(v_0,v_i)\to\infty$ as $i\to\infty$.
 Therefore, $diam_T(P(B_{\hat\lambda_i}))\to\infty$ as $i\to\infty$.
\end{proof}

\section{Hallways}\label{ht}

\begin{defn}(Page 45 of \cite{height})
 Let $P:X\to T$ be a tree of geodesic spaces.
 \begin{enumerate}
 \item (Hallways)
Let  $I=[0,1]$. A disk $\Theta:[0,n]\times I\to X$ is called to be a hallway of length $n$ if it satisfies the following conditions:
 \begin{enumerate}
 \item $\Theta^{-1}(\cup X_v: v\in Vertex(T))=\{0,1,...,n\}\times I$.
 \item $\Theta$ maps $i\times I$ to a geodesic in some vertex space $X_v$.
 \item $(P\circ \Theta): [0,n]\times I\to T$ factors through the canonical retraction to $[0,n]$
 and an isometry of $[0,n]$ to $T$. 
 \end{enumerate}
 \item (Boundary Thin Hallways) Let $\rho\geq 0$. A hallway $\Theta$ is said to be $\rho$-boundary thin if 
 $d(f(i,0),f(i+1,0))\leq \rho$ and $d(f(i,1),f(i+1,1))\leq \rho$ for all $i=0,1,...,n-1$.
  \end{enumerate}

\end{defn}

Given a geodesic $\hhat\lambda$ in $\hhat {X_{v_0}}$ and  $a\in B_{\hhat\lambda}\cap X$ there exists a $\rho$-quasi-isometric section
$r_a:[P(a),v_0]\to B_{\hhat \lambda}\cap X$ such that $r_a(P(a))=a$. For any element $x\in r_a([P(a),v_0])$,
we denote $\sigma^a_{\hhat\lambda}(x)$ to be the point $\hhat\lambda\cap r_a([P(a),v_0])$ in $X_{v_0}$.

\begin{lem}\label{hallway lemma} Suppose $\hhat{ X_{v_0}}$ is not quasiconvex in $\TT\CC(X)$. Then for each $i\in \mathbb N$, 
there exist a geodesic $\hhat\lambda_i\subset\hhat{X_{v_0}}$,
$a_i,b_i,x_i,y_i\in B_{\hhat {\lambda_i}}\cap X$ with $d_{\TT\CC(X)}(x_i,y_i)\leq 1$ and $P(x_i)=P(y_i)$ such that
$$d_{\hhat{X_{v_0}}}(\sigma^{a_i}_{\hhat\lambda_i}(x_i),\sigma^{b_i}_{\hhat\lambda_i}(y_i))\geq i.$$
 
\end{lem}
\begin{proof}
 Suppose the conclusion of the lemma does not hold. Then there exists $K\geq 0$ such that  for all geodesics $\hhat\lambda$ in $\hhat{X_{v_0}}$,
 for all points $a,b,x,y\in B_{{\hhat \lambda}}\cap X$ with $d_{\TT\CC(X)}(x,y)\leq 1$, $P(x)=P(y)$, 
  and for all quasi-isometric sections $r_{a}, r_{b}$ with $x\in r_{a}([P(a),v_0]), y\in r_{b}([P(b),v_0])$ 
 the distance 
 $$d_{\hhat{X_{v_0}}}(\sigma^{a}_{\hhat\lambda}(x),\sigma^{b}_{\hhat\lambda}(y)) \leq K.$$\\
 We will define a retraction map $\pi:B_{{\hhat \lambda}}\to\hhat{\lambda}$ such that it is uniformly coarsely Lipschitz.\\
  
 For all $x\in B_{\hhat \lambda}\cap X$, define $\pi(x)=\sigma^{x}_{\hhat\lambda}(x)$. If $x$ is a cone point of $B_{\hhat{\lambda}}$
 then there exists $x'\in B_{\hhat\lambda}\cap X$ such that $d_{\TT\CC(X)}(x,x')\leq \frac{1}{2}$ and $P(x)=P(x')$. 
 Define $\pi(x)=\sigma^{x'}_{\hhat\lambda}(x')$. Note that if there exists another $x''\in B_{\hhat\lambda}\cap X$
 such that $d_{\TT\CC(X)}(x,x'')\leq \frac{1}{2}$ and $P(x)=P(x'')$ then 
 from our assumption $d_{\hhat{X_{v_0}}}(\sigma^{x'}_{\hhat\lambda}(x'),\sigma^{x''}_{\hhat\lambda}(x'')) \leq K$.
 Thus, $\pi$ is well defined up to bounded discrepancy. \\
 By the same reason if $x,y\in B_{{\hhat \lambda}}$ with
 $P(x)=P(y)$ and $d_{\TT\CC(X)}(x,y)\leq 1$ then 
 $$d_{\hhat{X_{v_0}}}(\pi(x),\pi(y))\leq K.$$ Thus, by using triangle inequality, for all 
 $x,y\in B_{\hhat \lambda}$ with $P(x)=P(y)$ we have $$d_{\hhat{X_{v_0}}}(\pi(x),\pi(y))\leq K d_{\TT\CC(X)}(x,y)+K.$$\\
 Now, suppose that  $x,y\in B_{{\hhat \lambda}}\cap X$ and $d_T(P(x), P(y))=1$. Without any loss of generality we can assume that 
 $d_T(P(x),v_0)< d_T(P(y),v_0)$. Then, due to existence of quasi-isometric sections, there exists $z\in B_{{\hhat \lambda}}\cap r_y([P(y),v_0])$
 such that $P(x)=P(z)$ and $d_X(y,z)\leq \rho$. As above $d_{\hhat{X_{v_0}}}(\pi(x),\pi(z))\leq K d_{\TT\CC(X)}(x,z)+K$.\\
  As $z\in r_y([P(y),v_0])$, we have $\sigma^y_{\hhat\lambda}(y)=\sigma^y_{\hhat\lambda}(z)$. Again, by our assumption
 $$d_{\hhat{X_{v_0}}}(\sigma^y_{\hhat\lambda}(z),\sigma^z_{\hhat\lambda}(z))\leq K.$$ 
 So, $d_{\hhat{X_{v_0}}}(\pi(y),\pi(z))=d_{\hhat{X_{v_0}}}(\sigma^y_{\hhat\lambda}(z),\sigma^z_{\hhat\lambda}(z))\leq K$.
 And hence $$d_{\hhat{X_{v_0}}}(\pi(x),\pi(y))\leq K d_{\TT\CC(X)}(x,z)+2K\leq Kd_{\TT\CC(X)}(x,y) +\rho+2K.$$
 The above inequality implies that all geodesics $\hhat\lambda$ are uniformly quasiconvex in $B_{\hhat \lambda}$ and hence are uniformly quasiconvex in $\TT\CC(X)$.
 This statement contradicts the fact that  $\hhat{ X_{v_0}}$ is not quasiconvex in $\TT\CC(X)$.
 
\end{proof}
\begin{defn} (Quasi-isometric sections)
Let $P:X\to T$ be a tree of geodesic spaces.
Let $S$ be a geodesic segment in the tree $T$. A map $\Sigma: S\to X$ is said to be $\rho$-quasi-isometric section if $P(\Sigma(s))=s$ for all $s\in S$
and $d_S(u,u')\leq d_X(\Sigma(u),\Sigma(u'))\leq \rho d_S(u,u')$ for all $u,u'\in S$.
\end{defn}

\begin{defn}(Page 46 of \cite{height})(Hallway trapped by quasi-isometric sections) Let $P:X\to T$ be a tree of geodesic spaces and $\Theta$ be a hallway in $X$. Let $\mu_i=\Theta(i\times I)$,
and $\mu_i\subset X_{v_i}$, $i\in\{0,1,...,n\}$. 
We say that the hallway $\Theta$ with ends $\mu_0,\mu_n$ is trapped by two $\rho$-quasi-isometric sections $\Sigma_1:[v_0,v_n]\to X$, $\Sigma_2:[v_0,v_n]\to X$  if $\mu_i$ joins $\Sigma_1(v_i)$ to $\Sigma_2(v_i)$.

\end{defn}
Note that if a hallway $\Theta$ is trapped by $\rho$-quasi-isometric sections then $\Theta$ is  $\rho$-boundary thin.

The following lemma says that, if $\hhat{X_{v_0}}$ is  not a quasiconvex subset  in the  tree of coned-off spaces $\TT\CC(X)$, 
then there exists a sequence $\{\Theta_i\}$ of hallways  such that the lengths of $\Theta_i$ tend to 
infinity as $i\to\infty$. This fact for the tree of
 hyperbolic metric spaces was proved  by Mitra in \cite{height} 
 (viz. Lemma 4.2  and Corollary 4.3 of \cite{height} ). We adapt the arguments given in \cite{height}
in our set up of  relative hyperbolicity. 

\begin{lem}(Corollary 4.3 of \cite{height})\label{hallway existence}(Existence of hallways)
Suppose $P:X\to T$ is a tree of relatively hyperbolic spaces
with $\TC(X)$  a hyperbolic metric space.
Suppose $\hhat{X_{v_0}}$ is not quasiconvex
in $\TT\CC(X)$ then for each $i\in\mathbb N$ there exists a relative geodesic $\mu_i\subset X_{v_0}$  and
a $\rho$-boundary thin hallway $\Theta_i$ in $\TT\CC(X)$, with one  end as $\hhat\mu_i$, trapped by 
$\rho$-quasi-isometric sections $\Sigma_{1i},\Sigma_{2i}:P(\Theta_i)\to X$ such that the length of $\hhat\mu_i$ 
in $\hhat{X_{v_0}}$ is  greater than $i$ and the length of the hallway $\Theta_i$ is  a number $n_i$
such that $n_i\to\infty$ as $i\to\infty$.
\end{lem}
\begin{proof}
 
 For each $i\in\mathbb N$, from Lemma \ref{hallway lemma}, we have the following:\\
 (i) there exist a geodesic $\hhat\lambda_i\subset\hhat{X_{v_0}}$, $a_i,b_i,x_i,y_i\in B_{\hhat {\lambda_i}}\cap X$ with $d_{\TT\CC(X)}(x_i,y_i)\leq 1$, $P(x_i)=P(y_i)$, and\\ 
(ii) there exist $\rho$-quasi-isometric sections $r_{a_i}, r_{b_i}$ with $x_i\in r_{a_i}([P(a_i),v_0]), y_i\in r_{b_i}([P(b_i),v_0])$
\\ such that 
$$d_{\hhat{X_{v_0}}}(\sigma^{a_i}_{\hhat\lambda_i}(x_i),\sigma^{b_i}_{\hhat\lambda_i}(y_i))\geq i.$$
Let $\hhat\mu_i$ be the subsegment of $\hhat\lambda_i$ joining $\sigma^{a_i}_{\hhat\lambda_i}(x_i)$ and $\sigma^{b_i}_{\hhat\lambda_i}(y_i)$
and $\mu_i$ be a relative geodesic corresponding to $\hhat\mu_i$. Let $\Sigma_{1i}=r_{a_i}|_{[P(x_i),v_0]}$ and $\Sigma_{2i}= r_{b_i}|_{[P(y_i),v_0]}$
i.e. $\Sigma_{1i},\Sigma_{2i}$ are restrictions of $r_{a_i},r_{b_i}$ on $[P(x_i),v_0],[P(y_i),v_0]$ respectively.\\
As $P(x_i)=P(y_i)$ and $T$ is a tree, we have $[P(x_i),v_0]=[P(y_i),v_0]$.
Let $v_j\in [P(x_i),v_0] $ and $\hhat\mu_{ij}$ be a geodesic in $\hhat X_{v_j}$ joining $\Sigma_{1i}(v_j)$ and $\Sigma_{2i}(v_j)$
with $\hhat\mu_{i0}=\hhat\mu_i$.
Thus, we have a hallway $\Theta_i=\bigcup\limits_j{\hhat\mu_{ij}}$ trapped by quasi-isometric sections $\Sigma_{1i},\Sigma_{2i}$ and with one end $\hhat\mu_i$.
If lengths of the hallways $\Theta_i$ are uniformly bounded, then lengths of $\hhat\mu_i$ are also so, which is a contradiction.
Hence the statement of lemma holds.
\end{proof}

\section{Main Theorem}\label{mt}
Let $(\GG,\Lambda)$ be a finite graph of groups with all local groups finitely generated and  $P:X\to T$ be the tree of spaces corresponding to $(\GG,\Lambda)$. Let $G=\pi_1(\GG,\Lambda)$ and $\Sigma:[v_0,v_n]\to X$ be a quasi-isometric section, where $[v_0,v_n]$ is a geodesic in $T$.
The Cayley graph of the group $G$ is quasi-isometric to $X$. Thus, without loss of generality, we can take $\Sigma([v_0,v_n])$ to be contained in the Cayley graph of the group $G$.

\begin{defn}(Boundary Labeling of hallway) Let $\Theta$ be a hallway in the tree of spaces $P:X\rightarrow T$ coming from the finite graph of groups  $(\GG,\Lambda)$.
Let $\Theta$ be of length $n$ and it is trapped by quasi-isometric sections $\Sigma_1,\Sigma_2$. Let $v_0, v_1,...,v_n$  be successive vertices in $P(\Theta)\subset T$. We define  $Label(\Theta)$ as follows:\\
$$Label(\Theta)=({\Sigma_1(v_1)}^{-1}\Sigma_1(v_0),{\Sigma_2(v_{1})}^{-1}\Sigma_2(v_0),...,{\Sigma_1(v_n)}^{-1}\Sigma_1(v_{n-1}),{\Sigma_2(v_{n})}^{-1}\Sigma_2(v_{n-1})).$$
We say $Label(\Theta)$ to be the boundary labeling of $\Theta$.
 \end{defn}
 $Label$ is a map from the set of hallways of length $n$  to the Cartesian product $G^n$, where $G=\pi_1(\GG,\Lambda)$.  Let $\Theta$ be a hallway of length $n$ and it is trapped by $\rho$-quasi-isometric sections $\Sigma_1,\Sigma_2$. Then the length of the word $\Sigma_i(v_j)^{-1}\Sigma_i(v_{j-1})$ (in reduced form) is  bounded above by $\rho$, where $i=1,2$ and $j=1,...,n$. If $G$ is a finitely generated group, then there exists a finite subset $F$ of $G$ such that $Label(\Theta)\in F^n$ for all hallways  $\Theta$ of length $n$ . 
\begin{thm} \label{main}
Let  $G$ be the fundamental group of a finite  graph of groups  $(\GG,\Lambda)$
satisfying the following conditions:\\
(i) for each vertex $v$ of the finite graph $\Lambda$, the vertex group $G_v$ is finitely generated and hyperbolic relative to a 
finite collection $\HH_v$ of finitely generated subgroups of $G_v$;\\
(ii) the edge groups are quasi-isometrically embedded and relatively quasiconvex in respective vertex groups,  \\
(iii) $G$ is hyperbolic relative to  $\HH$, where $\HH$ is the collection of images of all $\HH_v$ in $G$,\\
(iv) for each vertex $v$ of $\Lambda$, $G_v$ has finite relative height in $G$.\\
Then for every vertex $w$ of $\Lambda$ the group $G_w$ is  a relatively quasiconvex subgroup of $G$.
\end{thm}
\begin{proof}
Let $v_0$ be a vertex of $\Lambda$. If possible, let $G_{v_0}$ be not relatively quasiconvex in $G$.  We will prove that there exists a 
vertex group $G_v$ and an
 infinite sequence $\{n_i\}$ of natural numbers such that the following holds:
 for each $n_i$ there exists $k_i=k(n_i)$-distinct cosets $g_{1n_i}G_v,...,g_{k_in_i}G_{v}$ such that $\bigcap\limits_{j=1}^{k_i}(g_{jn_i}G_vg_{jn_i}^{-1})$
 contains a loxodormic element and $k(n_i)\to\infty$ as $n_i\to\infty$.

 Finite subgroups and subgroups of parabolic subgroups are relatively quasiconvex.
 Now, as $G_{v_0}$ is not relatively quasiconvex in $G$, $G_{v_0}$ can not be a finite subgroup or contained in a conjugate of a parabolic
 subgroup $H\in\HH$. Let $X_{v_0}$ be the vertex space corresponding to the group $G_{v_0}.$
 Also, the space $\hhat{X_{v_0}}$ is not quasiconvex
 in $\TT\CC(X)$.  Using Lemma \ref{hallway existence}, for each $i\in\mathbb N$, there exists a relative geodesic $\mu_i\subset X_{v_0}$  and a
$\rho$-boundary thin hallway $\Theta_i$ in $\TT\CC(X)$, with one  end as $\mu_i$, trapped by 
$\rho$-quasi-isometric sections $\Sigma_{1i},\Sigma_{2i}:P(\Theta_i)\to X$ such that the length of $\hhat\mu_i$ 
in $\hhat{X_{v_0}}$ is  greater than $i$ and the length of the hallway $\Theta_i$ is a number $n_i$
such that $n_i\to\infty$ as $i\to\infty$. By taking all such hallways of length greater than $n_i$ and truncating them to of length $n_i$ gives infinitely such distinct hallways $\Theta_i$ of length $n_i$.

 Let $w_i$ be the element in $G_{v_0}$ representing $\mu_i$. If $w_i$ is a parabolic element then $w_i$ belongs to a parabolic subgroup and hence length of
$\hhat\mu_i$ is at most one. But, length of $\hhat\mu_i$ is greater than one. Thus, for each $i\in \mathbb N$, $w_i$  is a loxodormic element in $ G_{v_0}$. By our assumption, the embeddings of edge groups in respective vertex groups are strictly type preserving. This assumption implies that $w_i$ can never be a parabolic 
element of $G$. Hence, $w_i$ is also a loxodormic element in $G$.\par

  Consider $\mathcal S_{\rho}$ to be an infinite set of all $\rho$-boundary thin hallways of length $n$ emerging from $X_{v_0}$. 
 Let $q:T\to\Lambda$ be the quotient map. The restriction of composition $q\circ P$ to $\mathcal S_{\rho}$
 gives a map from the infinite set $\mathcal S_{\rho}$ to the finite set of vertices of $\Lambda$. 
 Thus, by Pigeon Hole Principle, there exists an infinite sequence $\{\Theta_m\}$ of hallways in $\mathcal S_{\rho}$ such that $q\circ P(\Theta_m)=q\circ P(\Theta_{m'})$
 for all $m\neq m'$. Consider this infinite sequence $\{\Theta_m\}$ of hallways where $q\circ P(\Theta_m)=q\circ P(\Theta_{m'})$
 for $m\neq m'$ and apply the $Label$ map on it. 
 $G$ being a finitely generated group, $Label(\{\Theta_m\})$ is a finite set. 
 Again, by Pigeon Hole Principle, there exists infinitely many hallways $\Theta_{m_j}\in\{\Theta_m\}\subset\mathcal S_{\rho}$ of same length
 with same  boundary labeling.\\ \\
 (Gluing of hallways:) Now consider two such $\rho$-boundary thin hallways $\Theta^1$, $\Theta^2$ of same length $n$, emerging from $X_{v_0}$, 
 with $q\circ P(\Theta^1)=q\circ P(\Theta^2)$  and of same boundary labeling. 
 Let $\Theta^i=\bigcup\limits_{j=1}^{n}\hhat\mu_{ij}$, $i=1,2$. Let $a_{ij},b_{ij}$ be end points of the relative geodesic $\mu_{ij}$. 
As $q\circ P(\Theta^1)=q\circ P(\Theta^2)$, for each $j\in\{1,2,...,n-1\}$ there exists a vertex $w_j\in\Lambda$ such that
$a_{1j}^{-1}b_{1j}$ and $a_{2j}^{-1}b_{2j}$ belongs to the group $G_{w_j}$. Due to same boundary 
 labeling, for all $j=1,2,...,n$, we have
 $$a_{1j+1}^{-1}a_{1j}=a_{2j+1}^{-1}a_{2j},$$
 $$ b_{1j+1}^{-1}b_{1j}=b_{2j+1}^{-1}b_{2j}.$$
 Let $c_{1j}=b_{1j}b_{2j}^{-1}a_{2j}$, then $a_{1j}^{-1}c_{1j}\in G_{w_j}$. Hence, $a_{1j}$ and $c_{1j}$ lie in same vertex space, say $X_{v_j}$,
 of the tree of space $X$. Let $\eta_j$
 be a relative geodesic in $X_{v_j}$ joining $a_{1j}$ and $c_{1j}$. Let $\Theta=\bigcup\limits_{j=1}^{n}\hhat\eta_{j}$, then $\Theta$ is a $\rho$-boundary thin hallway
 of length $n$ and $$a^{-1}_{1j+1}a_{1j}=c^{-1}_{1j+1}c_{1j}$$
 for all $j=1,2,...,n$.\\ \\
 
 Thus, by gluing hallways for each $n_i$ there exists a hallway $\Theta_{n_i}=\bigcup\limits_{j=1}^{n_i}\hhat\mu_{jn_i}$ of length $n_i$ such that\\
 (i) the relative geodesic $\mu_{1n_i}\subset X_{v_0}$\\
 (ii) if $p_{jn_i},q_{jn_i}$ are end points of $\mu_{jn_i}$
 then $$p^{-1}_{j+1n_i}p_{jn_i}=q^{-1}_{j+1n_i}q_{jn_i}$$
 for all $j=1,2,...,n_i-1$.
 
 Let $s_{jn_i}=p^{-1}_{j+1n_i}p_{jn_i}$ and $w_{jn_i}$ be the group element representing the relative geodesic $\mu_{jn_i}$.
 Then, for $j=1,...,n_i$
 $$w_{1n_i}=(s_{1n_i}s_{2n_i}...s_{jn_i})w_{jn_i}(s_{1n_i}s_{2n_i}...s_{jn_i})^{-1}.$$
  There exists a vertex $v_{jn_i}$ of the finite graph $\Lambda$
 such that $w_{jn_i}$ lies in the vertex group $G_{v_{jn_i}}$.

 Consider the set $S_{n_i}=\{w_{1n_i},...,w_{n_in_i}\}$.
 As $\Lambda$ is a finite graph,
 by Pigeon hole principle, there exists a vertex group $G_v$ and an infinite subsequence $\{k(n_i)\}$ of $\{n_i\}$ such that 
 $G_v\cap S_{n_i}$ contains $k(n_i)$ elements.
 Let $k_i=k(n_i)$.
 Suppose $G_v\cap S_{n_i}=\{w_{l_1n_i},w_{l_2n_i},...,w_{l_{k_i}n_i}\}$. 
 Then, $$w_{l_1n_i}=(s_{l_1n_i}...s_{l_jn_i})w_{l_jn_i}(s_{l_1n_i}...s_{l_jn_i})^{-1}.$$ 
 Let $g_{jn_i}=s_{l_1n_i}...s_{l_jn_i}$,
 Thus, $$w_{l_1n_i}\in\bigcap\limits_{j=1}^{k_i}g_{jn_i}G_vg_{jn_i}^{-1}.$$
 Thus, the intersection $\bigcap\limits_{j=1}^{k_i}g_{jn_i}G_vg_{jn_i}^{-1}$ contains 
 $w_{l_1n_i}$ which is a loxodormic element.
Thus, we have distinct left cosets $g_{1n_i}G_v,...,g_{k_in_i}G_v$ such that 
 $\bigcap\limits_{j=1}^{k_i}g_{jn_i}G_vg_{jn_i}^{-1}$ contains an loxodormic element and $k_i\to\infty$. So, relative height of $G_v$
 is infinite. This statement contradicts the hypothesis that relative height of each vertex group is finite.
 
\end{proof}

In the above Theorem \ref{main}, we can take relative height of each edge groups finite instead of taking relative height of each vertex groups finite.
Let $e_0$ be an edge between $v_0,v_0'$. Take a relative geodesic $\lambda_{e_0}$ in the edge space $X_{e_0}$. Join the  end points of
$f_{e_0,v_0}(\lambda_{e_0})$ by a relative geodesic $\lambda_{v_0}$ of $X_{v_0}$. Edge spaces being relatively quasiconvex in vertex spaces implies that
$\hhat{f_{e_0,v_0}(\lambda_{e_0})}$ lie in a uniform Hausdorff distance from $\hhat\lambda_{v_0}$. Construct the hyperbolic ladder $B_{\hhat\lambda_{v_0}}$
for $\hhat\lambda_{v_0}$. By the Flowing property \ref{Flow} of the hyperbolic ladder we have that 
if $v_1$ and $v_2$ are two adjacent vertices of $T$ joined by an edge $e$
with $B_{\hhat\lambda_{v_0}}\cap\hhat X_{v_1}=\hhat\lambda_{v_1}$ and 
$Nbhd_{G_v}(\lambda_{v_1}^b;C)\cap f_{e,v_1}(X_e)\neq\phi$ then $B_{\hhat\lambda_{v_0}}\cap\hhat X_{v_2}\neq\phi$. We take a maximal length
relative geodesic $\lambda_e$ in $X_e$ such that $f_{e,v_1}(\lambda_e)\subset Nbhd_{G_{v_1}}(\lambda_{v_1}^b;C)$.
These collection of relative geodesics $\lambda_e$ form a ladder $B_{\hhat\lambda_{e_0}}$ which are at a bounded Hausdorff distance from $B_{\hhat\lambda_{v_0}}$.
So, $B_{\hhat\lambda_{e_0}}$ is also  quasiconvex in $\TT\CC(X)$. If edge spaces are not relatively quasiconvex in $X$ then, as in Lemma \ref{hallway existence}, we get hallways
whose geodesics lie in edge spaces. Repeating the arguments in the proof of Theorem \ref{main}, by replacing vertex groups with edge groups,
we will get an edge group whose relative height is infinite.
 By taking the edge set of the finite graph to 
be  singleton, we have the following corollaries.

\begin{cor}\label{amalg cor}
(i) Let $G=A*_HB$, where $A,B$ are finitely generated groups with $H$ quasi-isometrically
embedded in both $A$ and $B$. Let $A_1,B_1$ be finitely generated subgroups of $A,B$ respectively such that $A$
is hyperbolic relative to $A_1$ and $B$
is hyperbolic relative to $B_1$. Suppose  $H$ is relatively quasiconvex in $A,B$ and $G$ is hyperbolic relative to $\{A_1,B_1\}$.
Then, relative height of $H$ is finite in $G$ if and only if
$H$ is relatively quasiconvex in $G$. \\ \\
(ii) Let $A$ be a  finitely generated group and $H$ be a subgroup of it. Let 
$\phi:H\to K$ be an isomorphism.  Let $A_1$ be a finitely generated subgroup of $A$ such that
$A$ is hyperbolic relative to subgroup $A_1$. Assume that $H$ is quasi-isometrically embedded and relatively quasiconvex in $A$.
Suppose the HNN-extension $G=<A,t|  tht^{-1}=\phi(h)~\forall ~h\in H>$ is hyperbolic relative  to $A_1$. 
Then, $H$ has finite relative height in $G$ if and only if $H$ is relatively quasiconvex in $G$.
\end{cor}

\section{Applications}\label{application}

An action of a group $G$ on a simplicial tree $T$ is said to be \textit{acylindrical} if there exists a constant $k\geq 1$ such that the pointwise
stabilizer of every geodesic segment in $T$ of length $\geq k$ is finite.

\begin{thm}
Let  $(\GG,\Lambda)$ be finite graph of groups such that the the action of the fundamental group of 
$(\GG,\Lambda)$ on the Bass-Serre covering tree $T$ of $\Lambda$ is acylindrical. Further,
 suppose the following conditions hold:\\
(i) for each vertex $v$ of the finite graph $\Lambda$, the vertex group $G_v$ is finitely generated and hyperbolic relative to a 
finite collection $\HH_v$ of finitely generated subgroups of $G_v$;\\
(ii) the edge groups are quasi-isometrically embedded and relatively quasiconvex in respective vertex groups,  \\
(iii) $G$ is hyperbolic relative to  $\HH$, where $\HH$ is the collection of images of all $\HH_v$ in $G$,\\
Then for every vertex $w$ of $\Lambda$ the group $G_w$ is  a relatively quasiconvex subgroup of $G$ .
\end{thm}
\begin{proof}
  If possible, suppose there exists an
 infinite sequence $\{n_i\}$ of natural numbers and a vertex $v$ of $\Lambda$ such that the following holds:
 for each $n_i$ there exists $k=k(n_i)$-distinct cosets $g_{1k}G_v,...,g_{kk}G_v$ such that $\bigcap\limits_{r=1}^{k}(g_{rk}G_vg_{rk}^{-1})$
 is infinite and $k(n_i)\to\infty$ as $n_i\to\infty$. Let $v_{1k},...,v_{kk}$ be the vertices in the Bass-Serre tree $T$ stabilized by
 $(g_{rk}G_vg_{rk}^{-1})$. The infinite subgroup 
 $\bigcap\limits_{r=1}^{k}(g_{rk}G_vg_{rk}^{-1})$ fixes the geodesic segment $[v_{1k},v_{kk}]$
 pointwise and $d_T(v_{1k},v_{kk})\to\infty$ as $k\to\infty$. This result contradicts the fact that the action of $G$ on $T$ is acylindrical. Thus, there exists a natural number $n$ such that if $m>n$ then for any distinct cosets  $g_1G_v,...g_kG_v$ the intersection
${\bigcap\limits_{i=1}^m}g_iG_vg^{-1}_i$ is finite for all vertex $v$ in $\Lambda$.
 Thus, in particular, each $G_v$ has finite relative height in $G$. By Theorem \ref{main}, we have the required result.
 
\end{proof}

\begin{defn}(Full Intersection with Parabolic Subgroups, Lemma 1.7 of \cite{dah}) A subgroup $H$ of a relatively hyperbolic group $(G,P)$ is said to 
 have full intersection with the conjugates of $P$ if $H\cap gPg^{-1}$ is either finite or has finite index in $gPg^{-1}$  for all $g\in G$.
\end{defn}
Dahmani, in \cite{dah}, introduced fully quasiconvex subgroups of a relatively hyperbolic group and he proved that a
fully quasiconvex subgroups have full intersection with parabolic subgroups.  Let $\Lambda(H)$ denote the limit set of $H$ in the relative hyperbolic boundary of $G$.

\begin{prop}\label{limit intersection}(Limit set Intersection, Proposition 1.10 of \cite{dah}) Let $Q_1,Q_2$ be relatively quasiconvex subgroups of a relatively hyperbolic group $G$
having full intersection with parabolic subgroups. Then $\Lambda(Q_1\cap Q_2)=\Lambda(Q_1)\cap\Lambda(Q_2)$.  
\end{prop}

Swarup, in \cite{swarup},  proved that  a finitely generated subgroup $Q$ of $G$ is quasiconvex if and only if 
the intersection of $Q$ with the conjugates of $A,B,$  in $G$ are quasiconvex (See Theorem 2 of \cite{swarup}). Theorem \ref{amalg thm} generalizes Swarup's result to relatively hyperbolic case. The proof of Theorem \ref{amalg thm} uses the same idea as that of Swarup's.

\begin{thm}\label{amalg thm}
Let  $G=A*_HB$, where $A,B,H$ are finitely generated groups with $H$ quasi-isometrically
embedded in $A,B$. Let $H_1$ be a finitely generated subgroup of $H$. Suppose $G,A,B,H$ are all hyperbolic
relative to the subgroup $H_1$ and $H$ is relatively quasiconvex in $G$. Suppose $Q$ is a finitely generated subgroup of $G$ containing $H_1$
such that $Q$ is hyperbolic relative to $H_1$ and $Q$ has full intersection with the conjugates of $H_1$.
The subgroup $Q$ is relatively quasiconvex in $G$ if and only if the intersection of $Q$ with every conjugates
of $A,B$ in $G$ are relatively quasiconvex in $G$.\\\\
\end{thm}

\begin{proof}
 The subgroup $H$ being relatively quasiconvex in $G$ implies that all conjugates of $A,B$ are also relatively quasiconvex in $G$. 
 Hence, if $Q$ is relatively quasiconvex in $G$ then the intersection of $Q$ with all the conjugates of $A,B$ are also so.
 For converse, let $T$ be the Bass-Serre tree and $X_T$ be the tree of relatively hyperbolic spaces corresponding to the amalgam   $G=A*_HB$.
 The quotient space $T/Q$ of $T$  gives a graph of groups structure for $Q$. Since $Q$ is finitely generated, there exists a finite subgraph $\Sigma$
 of $T/Q$ such that $Q$ is realized as a finite graph of groups with underlying graph $\Sigma$
 and vertex groups are intersection of $Q$ with some conjugates of $A,B$. Let $S$ be a component of inverse image of $\Sigma$ under the quotient map. Then
 $S$ is a subtree of $T$ with $S/Q=\Sigma$. We have a tree of spaces $Y_S$ corresponding to the  tree $S$ such that $Y_S$ naturally embeds in $X_T$. 
 The intersection of $Q$ with all conjugates of $A,B$ being relative quasiconvex  implies that they are relatively hyperbolic groups. 
 So, $Q$ turns out to be a finite graph-$\Sigma$ of relatively hyperbolic groups where
 edge groups (resp. coned-off edge groups) are
 quasi-isometrically embedded in vertex groups (resp. coned-off vertex groups) and hence
 $Y_S$ is a tree of relatively hyperbolic spaces. The subgroup $Q$  being hyperbolic relative to $H_1$ implies that the tree of coned-off spaces
 obtained by coning each local spaces of $Y_S$ is a hyperbolic metric space.
 Intersection of $Q$ with the conjugates of $A,B$, being relatively quasiconvex  in $G$,  have finite relative height in $G$ and hence in $Q$.
 Thus, by Theorem \ref{mt}, the intersection of $Q$ with the conjugates of $A,B$ are relatively quasiconvex in $Q$ and hence the vertex spaces of $Y_S$ are relatively quasiconvex in $Y_S$.\\ 
 Let $\partial_{rel}E$ denote the Bowditch boundary of a relatively hyperbolic group $E$. 
 Let $v$ be a vertex of $S$. As $S$ is a subtree of $T$, $v$ is also a vertex of $T$.
 Let $Q_v$ and $G_v$ denote the vertex groups for the groups $Q$ and $G$ respectively.
 The subgroups $Q_v$ and $G_v$ are relatively quasiconvex in the groups $Q$ and $G$ respectively. Therefore, $\partial_{rel} Q_v$   embeds in $\partial_{rel}Q$ and  $\partial_{rel} G_v$ embeds in $\partial_{rel}G$. The Bowditch
 boundaries $\partial_{rel}Q$ and $\partial_{rel}G$ have decomposition as follows: 
 $$\partial_{rel}Q=(\cup_{v\in V(S)}\partial_{rel} Q_v)\cup \partial S~\mbox{and}~ \partial_{rel}G=(\cup_{v\in V(T)}\partial_{rel} G_v)\cup \partial T .$$
 The subgroups $Q_v$ and $G_v$ are relatively quasiconvex in $G$ by our assumptions,  and $Q_v$ is a subgroup of $G_v$ preserving the cusps , therefore $Q_v$ is also relatively quasiconvex in $G_v$.
 Thus, there is a natural embedding
 $\partial_{rel} Q_v\hookrightarrow \partial_{rel} G_v$. The  subtree $S$ of $T$ induces an embedding $\partial S\hookrightarrow\partial T$.
 These embeddings induce a map $\partial_{rel}Q\to \partial_{rel}G$. We first prove that this map is injective.
The group $Q$  has full intersection with parabolic subgroups  implies that any vertex group $Q_v$ ($v$ is a vertex of $S$) has also full intersection with its parabolic subgroups. Thus, for any  two vertices $v_1,v_2$ of $S$, $\Lambda (Q_{v_1})\cap \Lambda (Q_{v_2})=\Lambda(Q_{v_1}\cap Q_{v_2})$ in both
 $\partial_{rel}Q$ and $\partial_{rel}G$.
  Hence, if two elements $\xi_1\in \partial_{rel} Q_{v_1} $ and $\xi_2\in \partial_{rel} Q_{v_2}$ are mapped to a same point in $\partial_{rel}G$
  then $\xi_1$ and $\xi_2$ are also identified in $\partial_{rel}Q$. \\ \\
  Now suppose if $Y_S$ is not relatively quasiconvex in $X_T$ then for each $n\in\mathbb N$ there exist $x_n,y_n\in Y_S$ and a relative geodesic $\gamma_n$ in $X_T$ joining $x_n$ and $y_n$ such that $\gamma_n$ does not lie in $n$-neighborhood of $Y_S$. Let   $\{[x_n,y_n]\}$ be a sequence of relative geodesics  in $Y_S$ joining $x_n$ and $y_n$. The action of the group $Q$ on $Y_S$ is proper and co-compact. Without loss of generality, by translating $[x_n,y_n]$ by a suitable element of $Q$,  we can assume that $[x_n,y_n]$ passes through a fixed ball $B(p;r)$ in $Y_S$ for all $n\in\mathbb N$.
  After passing to a subsequence, if necessary, $\{x_n\}$
  and $\{y_n\}$ converge to two different limits $\xi_1$ and $\xi_2$ respectively in $\partial_{rel}Q$.
  For all large $n$, $\gamma_n$ lie outside $n$-neighborhood of $B(p;r)$, which implies that 
  $\{x_n\}$  and $\{y_n\}$ converge to a same point in $\partial_{rel}G$. Thus,  
   $\xi_1$  gets identified with $\xi_2$ in $\partial_{rel}G$.
  This fact contradicts the injectivity of the map $\partial_{rel}Q\to \partial_{rel}G$. Hence $Y_S$ is relatively quasiconvex in $X_T$ and so
$Q$ is relatively quasiconvex in $G$.

\end{proof}

\end{document}